\documentclass[10pt]{article}
\usepackage{amsmath,amsfonts,dsfont}

\newtheorem{teor}{Theorem}
\newtheorem{coro}[teor]{Corollary}
\newtheorem{rem}[teor]{Remark}

\author{Magdalena Caballero\footnote{Corresponding author} and Rafael M. Rubio \\[6mm]
 Departamento de Matem\'aticas, Campus de Rabanales, \\[0.5mm] Universidad de C\'ordoba, 14071 C\'ordoba, Spain,\\[0.5mm] E-mails\textup{: \texttt{magdalena.caballero@uco.es},\quad  \texttt{rmrubio@uco.es}}}

\date{}

\begin{document}

\title{Characterizations of umbilic points of isometric immersions in Riemannian and Lorentzian manifolds\footnote{Running head: Characterizations of umbilic points}}

\maketitle

\thispagestyle{empty}

\begin{abstract} Several characterizations of umbilic points of 
submanifolds in arbitrary Riemannian and Lorentzian manifolds are 
given. As a consequence, we obtain new characterizations of spheres in the 
Euclidean space and of hyperbolic spaces in the Lorentz-Minkowski 
space. 

We also prove the Lorentzian version of a classical result by Cartan.

\end{abstract}


\section{Introduction} Umbilic points and totally umbilic submanifolds 
are topics of interest in both Riemannian and Lorentzian settings. 

Several 
authors have obtained different results on totally umbilic hypersurfaces of 
certain Riemannian manifolds. In this direction, M. Okumura in \cite{Ok} got 
sufficient conditions for a complete constant mean curvature hypersurface of a 
Riemannian manifold of non-negative constant curvature to be totally umbilic. 
Later, in \cite{Ha} T. Hasanis gave a result of the same kind for Riemannian 
manifolds of dimension at least four and positive constant curvature. 

In the Lorentzian setting, we can also find results providing sufficient conditions for a complete constant mean curvature spacelike hypersurface to be totally umbilic. In the de Sitter space, Goddard \cite{G} conjectured that every complete spacelike hypersurface with constant mean curvature must be totally umbilic. Although the conjecture turned out to be false, it motivated the search for a positive answer under appropriate additional hypotheses. This is the case of the works by K. Akutagawa \cite{Ak}, S. Montiel \cite{M}, who solved the conjecture in the compact case, and more recently by C. P. Aquino and H. F. de Lima \cite{AL}. The last author, jointly with F. Camargo, A. Caminha and U. Parente, studied the same problem in the anti-de Sitter space.
 
Without considering any hypothesis on the mean curvature, K. Nomizu and K. Yano 
in \cite{NY} characterized the connected totally umbilic submanifolds of a 
Euclidean space as the only submanifolds for which every circle of the 
submanifold is a circle of the space. More recently, in analogous terms, T. 
Adachi and S. Maeda in \cite{AM} characterized totally umbilic hypersurfaces in 
a Riemannian space form by a property of the extrinsic shape of circles on 
hypersurfaces. Finally, in \cite{CRu3,CRu}, the authors showed dual 
characterizations of the sphere in the Euclidean space and the hyperbolic space 
in the Lorentz-Minkowski space in dimension $3$ and in arbitrary dimension, making use of synthetic 
geometrical techniques. 

In this work we deal with the notion of umbilic point of an immersed 
submanifold in an arbitrary Riemannian or Lorentzian manifold. We obtain a 
characterization of this concept by the study of the intersections of the 
submanifold with normal submanifolds which are totally geodesic at this point, 
which provides a natural extension 
of the notion of umbilic point of a regular surface in the three-dimensional Euclidean 
space. In this direction, when the submanifold has codimension one, we also get 
a result on the mean curvature of the hypersurface at a given point. As an 
application, we prove new characterizations of the sphere in the Euclidean space, as well as of the hyperbolic space in 
the Lorentz-Minkowski space. In particular, we get that the $n$-sphere with 
$n\geq 3$ is the only hypersurface in the Euclidean space such that 
its intersections at each point with $2$ normal hyperplanes are 
$(n-1)$-spheres. Analogously, we prove that the $n$-hyperbolic space with $n\geq 
3$ is the only spacelike hypersurface in the Lorentz-Minkowski space such that 
its intersections at each point with $2$ normal hyperplanes are 
$(n-1)$-hyperbolic spaces. The first result is a wide generalization of Theorem 1 in \cite{CRu}.

We also prove the Lorentzian version of a classical result by Cartan. It asserts that for dimension bigger than 2 and fixed codimension r bigger than 1, the only Lorentzian manifolds for which for each point the existence of a totally geodesic r-submanifold with an arbitrary prescribed spacelike or timelike tangent space is assured, are those of constant sectional curvature.

The paper is organized as follows. In Section \ref{Preliminares} we present the 
basic background and we get the Lorentzian version of Cartan's result. The main results are shown in Section \ref{Main-Results}, and the computations 
needed for the proofs can be found in Section \ref{Set-up}. In the last section 
we apply the main results to the Euclidean space and the Lorentz-Minkowski 
space, in order to get characterizations of the sphere and the hyperbolic plane.

\section{Preliminaries and a first result} \label{Preliminares}
Let $(\overline{M}^{n+m},\overline{g})$ be a Riemannian manifold and consider an isometric immersion $x: (\Sigma^m,g)\longrightarrow (\overline{M},\overline{g})$. We denote by $\overline{\nabla}$ the Levi-Civita connection of $\overline{M}$ and by ${\nabla}^\Sigma$ the induced connection on $\Sigma$. 

Hence, the Gauss formula for $\Sigma$ reads as follows,
$$\overline{\nabla}_X Y={\nabla}^\Sigma_X Y+ {\rm II}_{\Sigma}(X,Y),$$
where ${\rm II}_{\Sigma}$ denotes the corresponding second fundamental form of $\Sigma$ and $X,Y$ are vector fields on $\Sigma$.

We define the mean curvature vector ${\rm H}(p)$ of the immersion at $p$ as 
$${\rm H}(p)=\dfrac{1}{m}{\rm trace\, II}_{\Sigma}(p).$$ 
If $\Sigma$ is a hypersurface, we can chose (at least locally) a unitary normal 
vector field to the hypersurface. Then, the second fundamental form is 
proportional to it, and so does the mean curvature vector field, whose 
proportionality constant is called the mean curvature of the hypersurface. 

The (immersed) submanifold is called totally geodesic at $p\in \Sigma$ if every geodesic $c_v(t)$ with initial vector $v\in T_p\Sigma$ is carried under $x$ into a geodesic of $\overline{M}$, for all $t$ in some neighborhood of $0$, \cite{Klingenberg}. The immersion $x$ is called totally geodesic if it is totally geodesic for all $p\in \Sigma$. Of course, if $x$ is totally geodesic at $p$, then its second fundamental form at $p$ vanishes. 

Whereas totally geodesic immersions of dimension bigger than $1$ and codimension 
bigger than $0$ are the exception, there always exist (local) submanifolds which 
are totally geodesic at a single point $p\in \overline{M}$ with an arbitrary 
prescribed tangent space $T\subset T_p\overline{M}$. Indeed, choose $r>0$ so 
small that the exponential map ${{\rm exp}_p}_{\mid B(0,r)}$ is injective. Then, 
${\rm exp}_p(B(0,r)\cap T)$ is such a submanifold. For dimension bigger than $2$ 
and fixed codimension bigger than $1$, the only manifolds for which for each 
point the existence of a totally geodesic submanifold with an arbitrary 
prescribed tangent space is assured, are those of constant sectional 
curvature. This result was proved by Cartan in \cite{Cartan}.

A point $p$ is called umbilic if its second fundamental form is equal to the 
metric tensor at $p$ times a normal vector field at the point $p$, the mean 
curvature vector field at $p$. In the particular case of immersed surfaces of 
codimension $1$, a point is umbilic if and only if the curvature of its 
normal sections coincide. The immersion $x$ is called totally umbilic if it is 
umbilic for all $p\in \Sigma$.

The contents of this section work the same for 
spacelike submanifolds of a Lorentzian manifold, that is, isometric immersions 
of a Riemannian manifold into a Lorentzian one. The two paragraphs on totally 
geodesic submanifolds also stay true for timelike submanifolds (isometric 
immersions of a Lorentzian manifold into another Lorentzian manifold). There 
are only two things to be remarked. The first one is that it is possible to 
choose $r>0$ small enough to assure that ${\rm exp}_p(B(0,r)\cap T)$ is a 
spacelike submanifold, if $T$ is made up of spacelike vectors, or a timelike 
submanifold if $T$ is a Lorentzian subspace. The second one is the Lorentzian 
version of Cartan's result, which constitutes our first result.

Before stating the result, we need some basic preliminaries on 
Lorentzian geometry. Let $(\overline{M}^{n+1},\overline{g})$ be a 
Lorentzian manifold and  $p\in\overline{M}$. $X\in T_p\overline{M}-\{0\}$ is 
called spacelike if $\overline{g}(X,X)>0$, timelike if $\overline{g}(X,X)<0$ and 
light-like if $\overline{g}(X,X)=0$. An immersed submanifold of $\overline{M}$ 
is called non-degenerate if its induced metric is either Riemannian or 
Lorentzian.

As well as the Codazzi 
equation for a non-degenerate submanifold of $\overline{M}$, $\Sigma$,
$$(R(X,Y)Z)^{\bot}=(\nabla^{\bot}_{X}{\rm 
II}_{\Sigma})(Y,Z)-(\nabla^{\bot}_{Y}{\rm II}_{\Sigma})(X,Z),$$
where $R$ is the curvature tensor of 
$\overline{M}$, $X,Y,Z$ are vector fields on $\Sigma$, $\nabla^{\bot}$ is the 
normal connection of $\Sigma$ and 
$$(\nabla^{\bot}_{X}{\rm 
II}_{\Sigma})(Y,Z)=\nabla^{\bot}_{X}{\rm 
II}_{\Sigma}(Y,Z)-{\rm II}_{\Sigma}(\nabla_X Y, Z)-{\rm 
II}_{\Sigma}(Y,\nabla_X Z).$$

\begin{teor}\label{t1}
For dimension bigger than $2$ and fixed codimension $r$ bigger than $1$, the 
only Lorentzian manifolds for which for each point the existence of 
a totally geodesic $r$-submanifold with an arbitrary prescribed spacelike or 
timelike tangent space is assured, are those of constant sectional curvature.
\end{teor}

The proof follows the ideas in \cite{Daczer}.

\begin{proof}
Let $(\overline{M}^{n+1},\overline{g})$ be a Lorentzian manifold with $2\leq n$ and fixed 
$2\leq r \leq n$  in such a way such that for each point the existence of a totally geodesic 
$r$-submanifold with an arbitrary prescribed spacelike or timelike tangent 
space is assured. 

From Schur's lemma, see \cite{O-N}, it is enough to show that at each point the 
sectional curvature is constant on the non-degenerate planes of the tangent 
space. 

We fix $p\in\overline{M}$ and consider orthonormal $x,y,z\in T_p \overline{M}$. 
There exists a totally geodesic $r$-submanifold such that $x,y$ belong to its 
tangent plane at $p$ and $z$ is orthogonal to it. Since the submanifold is 
totally geodesic, its second fundamental form vanishes, and the Codazzi 
equation assures us that $\overline{g}(R(x,y)z,x)=0$.

Let $x,y,z\in T_p \overline{M}$ 
be spacelike orthonormal vectors. 
Then $$K(x,y)-K(x,z)=2 \,\overline{g}(R(x,y')z',x)=0,$$
where $K(x,y)$ is the sectional curvature of the 
plane spanned by $x$ and $y$, $y'=\dfrac{1}{\sqrt{2}}(y+z)$ and 
$z'=\dfrac{1}{\sqrt{2}}(y-z)$.

Let $x,y,z\in T_p \overline{M}$ 
be orthonormal vectors, such that $x,y$ are spacelike and $z$ is timelike. 
Then $$K(x,y)-K(x,z)=\dfrac{2}{\sqrt{3}} \,\overline{g}(R(x,y')z',x)=0,$$
where $y'=\sqrt{1/2} \, y+\sqrt{3/2} \, z$ 
and $z'=\sqrt{3/2} \, y+\sqrt{1/2} \, z$.

The proof ends with the following two statements.

For each two non-degenerate planes, there exists a chain of planes from one to the other, such that in pairs they span a Lorentzian $3$-dimensional space. 

For each pair of non-degenerate planes in a Lorentzian space of dimension $3$, there exists a chain of non-degenerate planes from one to the other, such that in pairs they intersect orthogonally and both planes of each pair are spacelike, or one is spacelike and the other timelike.

To prove the first statement, it is enough to consider the case of a Lorentzian plane, $\Pi_1$, and an spacelike plane, $\Pi_2$. Take $\{u_1,u_2\}$ an orthonormal base of $\Pi_1$ such that $u_1$ is timelike, $u_3$ a unitary vector belonging to $\{u_2\}^{\bot}\cap\Pi_2$ and $\{u_4,u_5\}$ an orthonormal base of $\Pi_2$ such that $\overline{g}(u_1,u_4)=0$. It is clear that $\Pi_3={\rm span}\{u_1,u_3\}$ is a timelike plane. It is easy to check that $\Pi_1+\Pi_3$ and $\Pi_2+\Pi_3$ are non-degenerate timelike $3$-dimensional spaces.

To prove the second one, for each pair of non-degenerate planes in a $3$-dimensional space, $\Pi_1$ and $\Pi_2$, we consider the plane generated by their orthogonal vectors, $\Pi_{1,2}$. We need to distinguish cases. If both planes are spacelike, $\Pi_{1,2}$ is timelike, and we have finished. If both planes are timelike, $\Pi_{1,2}$ can be spacelike or degenerate, in the last case, we can take $\Pi_4$ timelike such that $\Pi_{1,4}$ and $\Pi_{2,4}$ are spacelike. If one plane is timelike and the other is spacelike, $\Pi_{1,2}$ is timelike. Applying the previous case, we finish the proof.

\end{proof}

\section{Set up}\label{Set-up}

Let $(\overline{M}^{n+m},\overline{g})$ be a $n+m$-dimensional Riemannian 
manifold and $x:\Sigma^m\longrightarrow\overline{M}$ an isometrically immersed 
submanifold.  For all $q\in \Sigma$, we can consider the normal space to 
$\Sigma$ at $q$ by $T_q^{\bot}\Sigma$. We take $s\leq m-1$ and we denote by $M$ a 
$(s+n)$-submanifold in $\overline{M}$ such that it is totally geodesic at $q$ 
and $T_q^{\bot}\Sigma \subset T_qM$.

Consider $S_M=M\cap \Sigma$ and denote by ${\nabla}^M$ and ${\nabla}^{S_M}$ the induced connections on $M$ and $S_M$, respectively.

Let $X,Y$ be two vector fields on $S_M$. Taking into account that $M$ is totally geodesic at $q$, we have 
\begin{equation}\label{igual}
\overline{\nabla}_X Y(q)={\nabla}^M_X Y(q).
\end{equation}

\noindent Moreover, the respective orthogonal projections satisfy
\begin{align}\label{tan}
{\nabla}_X^\Sigma Y(q)&={\rm Tan}_{\mid\Sigma}(\overline{\nabla}_X Y(q))={\rm 
Tan}_{\mid\Sigma}({\nabla}^M_X Y(q))
\nonumber\\
&={\rm Tan}_{\mid S_M}({\nabla}^M_X 
Y(q))={\rm Tan}_{\mid S_M}(\overline{\nabla}_X Y(q))={\nabla}^{S_M}_X Y(q).
\end{align}

\noindent Hence, by comparing the Gauss formula for $\Sigma$ and for $S_M$, we deduce the following equality
\begin{equation}\label{operator}
{\rm II}_{\Sigma}(X,Y)(q)={\rm II}_{S_M}(X,Y)(q).
\end{equation}

\section{Main Results}\label{Main-Results}

When studying regular surfaces in $\mathds{R}^3$, we see that the usual definition of umbilic point states that a point is umbilic if the principal curvatures of the surface at the point coincide, \cite{Do-Carmo}, equivalently, if the curvature of all the normal sections to the surface at the point coincide.   
Motivated by this definition, we give the following result.

\begin{teor}\label{t2}
Let $\Sigma$ be an $m$-dimensional submanifold immersed in a Riemannian manifold $\overline{M}^{m+n}$ with $m\geq 2$ and take $q\in\Sigma$. For a 
fixed $0<s\leq m-1$, we assume that for each $(s+n)$-submanifold $M$ such that it 
is totally geodesic at $q$ and $T_q^{\bot}\Sigma \subset T_qM$, we get the same 
mean curvature vector at $q$ as that of the isometrically immersed $s$-submanifold 
$S_M:=M\cap \Sigma$. Then, the point  $q\in\Sigma$ is umbilic. 

\end{teor}

\begin{proof} For any two arbitrary unitary vectors $u,u'\in T_q\Sigma$, take 
$u_2, ..., u_{s}\in T_q\Sigma$ such that $\{u, u_2, ..., u_s\}$ and $\{u', 
u_2, ..., u_s\}$ are orthonormal sets, and consider the totally umbilic 
submanifolds at $q$ spanned by them and $T_q^{\bot}\Sigma$, $M$ and $M'$ 
respectively.

Our hypothesis guaranty that

$${\rm II}_{S_M}(u,u) + \sum {\rm II}_{S_M}(u_i,u_i)={\rm II}_{S_{M'}}(u',u') + \sum {\rm II}_{S_{M'}}(u_i,u_i),$$

\noindent and as a direct consequence of (\ref{operator}) we obtain ${\rm II}_{\Sigma} (u,u)= {\rm II}_{\Sigma} (u',u')$.

\end{proof}

Note that for $m=2$ and codimension $1$, Theorem \ref{t2} holds trivially. In this case, the curves $S_M$ must be considered as 1-dimensional submanifolds immersed in $\Sigma$.

When we focus on the case of codimension $1$, thanks to (\ref{igual}) and 
(\ref{tan}), by using the Gauss formula for $S_M$ in $M$, we can deduce that
\begin{equation}\label{operator2}
\overline{g}(D_M(X),Y)(q)=\overline{g}(AX,Y)(q),
\end{equation}
for any two tangent vector fields of $S_M$, $X$ and $Y$, where $A$ is the Weingarten operator of $\Sigma$ in $\overline{M}$ and $D_M$ is the Weingarten operator of $S_M$ in $M$.
 
Now fix $0<s\leq m-1$ and take the orthonormal base $\{e_1,...,e_{m}\}$ of $T_q\Sigma$ consisting of 
the eigenvectors of the Weingarten operator of $\Sigma$. We define the 
submanifolds  $\{M_{\alpha}\}_{\alpha}$ as the totally geodesic at $q$ 
submanifolds spanned by the unitary normal vector field to the hypersurface 
$\Sigma$ at $q$, $N(q)$, and any subset of $\{e_1,...,e_{m}\}$ with cardinal 
$s$. In this setting, Theorem \ref{t2} reads as follows.

\begin{coro}\label{c3}
Let $\Sigma$ be a hypersurface of a Riemannian manifold $\overline{M}$ with dimension at least $3$, $q\in\Sigma$ and suppose that the mean curvature of $S_{M_\alpha}$ at $q$ in $M_{\alpha}$ does not depend on $\alpha$. Then, the point  $q\in\Sigma$ is umbilic and its mean curvature is that of each $S_{M_\alpha}$.
\end{coro}

 Note that this last criterion  requires its checking on a finite number of 
submanifolds, contrary to Theorem \ref{t2}, which requires a fine behavior for 
all hypersurface $M$ satisfying the conditions imposed.

\begin{rem}\label{r4}
The previous corollary does not hold if we change the base of eigenvectors by 
any orthonormal base. Indeed, take a surface of $\mathds{R}^3$ admitting a zero 
mean curvature non umbilic point, for instance the hyperbolic paraboloid. The 
asymptotic directions at this point give us an orthonormal base for which the 
normal sections have zero curvature at the point. On the contrary, the thesis 
concerning the mean curvature of the point remains true.
\end{rem}

Thanks to (\ref{operator2}), a generalization of the last statement of the previous remark can be given, which constitutes a natural characterization of the mean curvature of a hypersurface at a point. 

\begin{coro}\label{c5} Let $\Sigma$ be a hypersurface of a Riemannian manifold $\overline{M}$  with dimension at least $3$, take $q\in\Sigma$ and 
an arbitrary orthonormal base $\{v_1,...,v_m\}$ of $T_q\Sigma$. For a fixed 
$0<s\leq m-1$, consider the submanifolds $\{M_{\alpha}\}_{\alpha}$ which are 
totally geodesic at $q$ and are spanned by $N(q)$ and any subset of 
$\{v_1,...,v_m\}$ with cardinal $s$. Then, the mean curvature of $\Sigma$ at 
$q$ is the mean of the mean curvature of the hypersurfaces 
$S_{M_{\alpha}}=M_\alpha\cap\Sigma$ of $M_\alpha$ at $q$.
\end{coro}

On the other hand, the converse of Theorem \ref{t2} holds.

\begin{teor}\label{t6} Let $\Sigma$ be an $m$-dimensional submanifold of 
the Riemannian manifold $\overline{M}^{m+n}$ with $m\geq 2$ and let $q\in\Sigma$ be a umbilic point. 
Fix $0<s\leq m-1$. Then, the mean curvature vector field at $q$ of the 
submanifold $\Sigma\cap M$ coincides with the mean curvature vector field of 
$\Sigma$ at $q$, for each $(s+n)$-submanifold $M$ such that it 
is totally geodesic at $q$ and $T_q^{\bot}\Sigma \subset T_qM$.
\end{teor}

\begin{proof}
The proof follows from (\ref{operator}) and the definition of the mean 
curvature vector. 

\end{proof}

\vspace{2mm}

As well as the converse of Corollary \ref{c3}.

\begin{coro}\label{c7} Let $\Sigma$ be a hypersurface of 
the Riemannian manifold $\overline{M}$  with dimension at least $3$ and let $q\in\Sigma$ be a umbilic point. 
Then, the mean curvature of $S_{M_{\alpha}}$ at $q$ in $M_{\alpha}$ coincides 
with the mean curvature of $\Sigma$ at $q$, for all $M_{\alpha}$.
\end{coro}

And finally, as a consequence of (\ref{operator}), we get two more results. They are the first one for submanifolds and the second one for hypersurfaces.

\begin{teor}\label{t8}Let $\Sigma$ be an $m$-dimensional submanifold of a Riemannian manifold $\overline{M}^{m+n}$ with 
$m\geq 3$ and 
take $q\in\Sigma$. Fix $2\leq s\leq m-1$ and assume that for each $(s+n)$-submanifold $M$ such 
that it 
is totally geodesic at $q$ and $T_q^{\bot}\Sigma \subset T_qM$, the point $q$ 
is umbilic in $S_M:=M\cap \Sigma$ as a submanifold of $\overline{M}$ 
(equivalently as submanifold of $M$). Then, the point  $q\in\Sigma$ is umbilic.
\end{teor}

\begin{proof}
Consider equation (\ref{operator}) for a submanifold $S_M$ in 
$\overline{M}$. Then
$${\rm II}_{\Sigma}(X,Y)(q)={\rm II}_{S_M}(X,Y)(q),$$
\noindent where $X,Y\in \mathfrak{X}(S_M)$. If we suppose that the point $q\in S_M$ is 
umbilic, then

\begin{equation*}
{\rm II}_{S_M}(X,Y)(q)=\overline{g} (X,Y) (q) \mathrm{H}(q),
\end{equation*}

\noindent where $\mathrm{H}(q)$ denotes the mean curvature vector field of 
$S_M$ at $q$.

Thanks to (\ref{operator}), $\mathrm{H}(q)$ does not depend on $M$, and so, $q$ 
 is also umbilic in $\Sigma$.

\end{proof}

In fact, the hypothesis of the previous theorem can be weakened.

\begin{rem}\label{r9}
Consider an arbitrary base $\{w_1,...,w_m\}$  for 
$T_q\Sigma$ and take the submanifolds 
$\{M_{\alpha}\}_{\alpha}$ which are totally geodesic at $q$ and are spanned by 
$T_q^{\bot}\Sigma$ and any subset of $\{w_1,...,w_m\}$ with cardinal $s\geq 2$.
If we suppose that the point $q$ is umbilic in $S_{M_ {\alpha}}$ for all 
$\alpha$, then $q$ is also umbilic in $\Sigma$.
\end{rem}

In the particular case of hypersurfaces cut by hypersurfaces, only two cuts are needed.

\begin{teor}\label{t10}
Let $\Sigma$ be a hypersurface of a Riemannian manifold $\overline{M}^{m+1}$ with 
$m\geq 3$ and 
take $q\in\Sigma$. If there exist two different normal hypersurfaces, $M_1$ and $M_2$, totally geodesic at $q$ and such that $q$ is umbilic in $S_{M_i}:=M_i\cap \Sigma$ as a submanifold of $\overline{M}$ 
(equivalently as submanifold of $M_i$) for $i=1,2$, then the point  $q\in\Sigma$ is umbilic.

\end{teor}

Observe that the converse of Theorem \ref{t8} and that of Theorem \ref{t10} hold  
trivially.

\subsection{Spacelike hypersurfaces in Lorentzian manifolds}

Suppose now that $(\overline{M}^{n+m}, \overline{g})$ is a Lorentzian manifold and  
$\Sigma^m$ is a spacelike submanifold of $\overline{M}$. Let 
$q\in\Sigma$ be an arbitrary point and let $T\subset T_q\Sigma$ be a subspace of $T_q\Sigma$. Making use of the exponential map, we 
can find a Lorentzian submanifold $M$ in $\overline{M}$ through $q$, with $T_q\Sigma={\rm 
span}\big( T\cup T_q^{\bot}\Sigma\big)$ such that $M$ is totally geodesic at $q$. 
Thus, the submanifold $S_M=M\cap\Sigma$ is spacelike. 

In this setting, the set up and all the results of this section remain true.

\section{Several consequences}\label{Several-Consequences}

\subsection{For submanifolds of the Euclidean space}

In the sequel, we say that an affine subspace of the Euclidean space through a point is normal to a submanifold at that point, if it contains the normal subspace to the submanifold at the point.

A first consequence is obtained from Theorem \ref{t8} and Remark \ref{r9}.

\begin{coro}\label{c11} Let $\Sigma$ be an m-dimensional submanifold of $\mathbb{R}^ {m+n}$ ($m\geq 3$), $q\in \Sigma$ and $2\leq s\leq m-1$ fixed. Consider an arbitrary base of $T_q\Sigma$ and the family consisting of all the $(s+n)$-dimensional normal affine subspaces through $q$ spanned by the vectors of the base, $\{\Pi_i\}_{i\in I}$. If $q\in \Sigma\cap \Pi_i$ is umbilic in $\Pi_i$ for all $i\in I$, then $q$ is a umbilic point of $\Sigma$.
\end{coro}

\noindent In particular, we can state the following characterization.

\begin{coro}\label{c12} Let $\Sigma$ be an m-dimensional submanifold of $\mathbb{R}^ {m+n}$ ($m\geq 3$), $q\in \Sigma$ and $2\leq s\leq m-1$ fixed. Consider an arbitrary base of $T_q\Sigma$ and the family consisting of all the $(s+n)$-dimensional normal affine subspaces through $q$ spanned by the vectors of the base, $\{\Pi_i\}_{i\in I}$. If for each $i\in I$ the intersection $\Sigma\cap \Pi_i$ is a piece of an $s$-sphere, then $q$ is a umbilic point of $\Sigma$.
\end{coro}

We deal now with totally umbilic submanifolds. An $m$-dimensional non-totally geodesic, totally umbilic submanifold of a Euclidean $(m+n)$-space is contained in a sphere of an affine $(m+1)$-subspace, \cite{B-Y-Chen}. Hence, from now on we can focus on hypersurfaces. 

As a consequence of Remark \ref{r9} we can also enunciate a characterization of the sphere in the Euclidean space of dimension at least four.

\begin{coro}\label{c13} The sphere $\mathbb{S}^m$ is the only connected and complete hypersurface of $\mathbb{R}^{m+1}$ ($m\geq 3$) such that, for some $2\leq s\leq m$ fixed, its intersection at each point $q$ with the $(s+1)$-dimensional normal affine subspaces through $q$ spanned by an arbitrary base of the tangent space at $q$, are pieces of spheres.
\end{coro}

In the previous result, the completeness can be omitted if we ask the intersections to be spheres. In the particular case of hypersurfaces, Theorem \ref{t10} gives us the next result, in which we have not included the completeness as an hypothesis.

\begin{coro}\label{c14} The only connected hypersurface $\Sigma$ in $\mathbb{R}^{m+1}$ ($m\geq 3$) such that at each point $q\in \Sigma$ its intersection with $2$ normal hyperplanes are spheres, is the sphere $\mathbb{S}^m$.
\end{coro}

The connectedness can be omitted as an hypothesis. Assume $\Sigma$ is not connected. Then we can apply the previous result to each connected component. We get that $\Sigma$ is made up of spheres. Then, there are points of $\Sigma$ such that for any normal hypersurface at the point, the intersection with $\Sigma$ is not an sphere, but many. Which is a contradiction.

We get a result that constitutes a wide generalization Theorem 1 of \cite{CRu}.

\begin{coro}\label{c15} The only hypersurface $\Sigma$ in $\mathbb{R}^{m+1}$ ($m\geq 3$), such that at each point $q\in \Sigma$ its intersection with $2$ normal hyperplanes 
are spheres, is the sphere $\mathbb{S}^m$.
\end{coro}

\noindent Making use of Alexandrov's theorem \cite{A} and Corollary \ref{c5}, we can include the case $m=2$ in the following result.

\begin{coro}\label{c16} The only compact and connected  hypersurface $\Sigma$ in $\mathbb{R}^{m+1}$ such that its intersection at each point $q\in \Sigma$ with $2$  orthonormal normal hyperplanes (i.e., its normal vectors at $q$ are orthonormal) 
are pieces of  spheres with radius $r$, is a  sphere with radius $r$.
\end{coro}

\subsection{For submanifolds of the Lorentz-Minkowski space}

The second part of the section is devoted to the Lorentz-Minkowski space. Taking into account Section 2.1, we obtain results for spacelike submanifolds analogous to those in the Euclidean space. 

\begin{coro}\label{c17} Let $\Sigma$ be an m-dimensional spacelike submanifold of $\mathbb{L}^ {m+n}$ ($m\geq 3$), $q\in \Sigma$ and $2\leq s\leq m-1$ fixed. Consider an arbitrary base of $T_q\Sigma$ and the family consisting of all the $(s+n)$-dimensional normal affine subspaces through $q$ spanned by the vectors of the base, $\{\Pi_i\}_{i\in I}$. If $q\in \Sigma\cap \Pi_i$ is umbilic in $\Pi_i$ for all $i\in I$, then $q$ is a umbilic point of $\Sigma$.\end{coro}

\noindent In particular, we can state the following characterization.

\begin{coro}\label{s4}
Let $\Sigma$ be an m-dimensional spacelike submanifold of $\mathbb{L}^ {m+n}$ ($m\geq 3$), $q\in \Sigma$ and $2\leq s\leq m-1$ fixed. Consider an arbitrary base of $T_q\Sigma$ and the family consisting of all the $(s+n)$-dimensional normal affine subspaces through $q$ spanned by the vectors of the base, $\{\Pi_i\}_{i\in I}$. If for each $i\in I$ the intersection $\Sigma\cap \Pi_i$ is a piece of an $s$-hyperbolic plane, then $q$ is a umbilic point of $\Sigma$.
\end{coro}

When studying totally umbilic submanifolds, it is possible to focus on the case of hypersurfaces. The reason is that an $m$-dimensional spacelike non-totally geodesic, totally umbilic submanifold of the Lorentz-Minkowski $(m+n)$-space is either contained in a hyperbolic space of an affine timelike $(m+1)$-subspace or in a sphere of an affine spacelike $(m+1)$-subspace, \cite{AKK}.

\begin{coro}\label{c19} The hyperbolic space $\mathds{H}^m$ is the only connected and complete spacelike hypersurface of $\mathbb{L}^{m+1}$ ($m\geq 3$) such that, for some  $2\leq s\leq m$ fixed, its intersection at each point $q$ with the $(s+1)$-dimensional normal affine subspaces through $q$ 
spanned by an arbitrary base of the tangent space at $q$, are pieces of hyperbolic spaces.
\end{coro}

Reasoning as in the Euclidean case, we get the following result.

\begin{coro}\label{c20} The only spacelike hypersurface $\Sigma$ in $\mathbb{L}^{m+1}$ ($m\geq 3$) such that its intersection at each point $q\in \Sigma$ with two normal hyperplanes are $(m-1)$-hyperbolic spaces, is the hyperbolic $m$-space.
\end{coro}

Observe that if we consider the standard realization of the hyperbolic space $\mathds{H}^n(1)$ given by the upper component of the hyperboloid of two sheets, as a direct consequence of our result it is clear that the intersection of any vectorial timelike hyperplane with $\mathbb{H}^n(1)$, must be a $(n-1)$-hyperbolic space, with mean and sectional curvature equal to $1$, i.e, a copy of $\mathds{H}^{n-1}(1)$ embedded in $\mathbb{L}^{n+1}$ and totally umbilic.

\section*{Acknowledgments}

The  authors are partially supported by the Spanish MICINN Grant with FEDER funds MTM2013-47828-C2-1-P.


\begin{thebibliography}{99}

\bibitem{AM} T. Adachi and S. Maeda, Characterization of totally umbilic hypersurfaces in a space form by circles, \emph{Czechoslovak Math. J.} 55 (2005), no. 130, 203--207.

\bibitem{AKK} S.-S. Ahn, D.-S. Kim and Y. H. Kim, Totally Umbilic Lorentzian Submanifolds, \emph{J. Korean Math. Soc.} 33 (1996), 507--511.

\bibitem{Ak} K. Akutagawa, On spacelike hypersurfaces with constant mean curvature in the de Sitter space, \emph{Math. Z.} 196 (1987), 13--19.

\bibitem{A} A. Alexandrov, Uniqueness theorems for surfaces in the large, \emph{Vestnik Leningrad Univ.} 13 (1958), 5--8. 

\bibitem{AL} C. P. Aquino and H. F. de Lima, On the umbilicity of complete constant mean curvature spacelike hypersurfaces, \emph{Math. Ann.} 360 (2014), 555--569.

\bibitem{CRu3} M. Caballero and R.M. Rubio, A dual rigidity of the sphere and the hyperbolic plane, \emph{preprint.}

\bibitem{CRu} M. Caballero and R.M. Rubio, Dual characterizations of the sphere and the hyperbolic space in arbitrary dimension, \emph{Int. J. Geom. Methods Mod. Phys.} 12 (2015), 5 pp.

\bibitem{CCLP} F. Camargo, A. Caminha, H. F. de Lima and U. Parente, Generalized maximum principles and the rigidity of complete spacelike hypersurfaces, \emph{Math. Proc. Camb. Philos. Soc.} 153 (2012), 541--556.

\bibitem{Do-Carmo} M. P. do Carmo, \emph{Differential Geometry of Curves and Surfaces,} Prentice-Hall, 1973.

\bibitem{Cartan} E. Cartan, \emph{Le\c{c}ons sur la G\'eom\'etrie des Espaces de Riemann}, Gauthier-Villars, 1951.

\bibitem{B-Y-Chen} B.-Y. Chen, \emph{Total mean curvature and submanifolds of finite type}, World scientific, 1984.

\bibitem{Daczer} M. Daczer, \emph{Submanifolds and isometric immersions}, 
Mathematics Lecture Series, 13, Publish or Perish Inc., 1990.

\bibitem{G} A. J. Goddard, Some remarks on the existence of spacelike hypersurfaces of constant mean curvature, \emph{Math. Proc. Camb. Philos. Soc.} 82 (1977), 489--495.

\bibitem{Ha} T. Hasanis, Characterization of totally umbilical hypersurfaces, \emph{Proc. Am. Math. Soc.} 81 (1981), no. 3, 
447--550.

\bibitem{Klingenberg} W. P. A. Klingenberg, \emph{Riemannian Geometry}, de Gruyter, 1995. 

\bibitem{M} S. Montiel, An integral inequality for compact spacelike hypersurfaces in de Sitter space and applications to the case of constant mean curvature, \emph{Indiana Univ. Math. J.} 37 (1988), 909--917.
	
\bibitem{NY} K. Nomizu and K. Yano, On circles and spheres in Riemannian geometry, \emph{Math. Ann.} 210 (1974), 163--170.

\bibitem{Ok} M. Okumura, Hypersurfaces and a pinching problem on the second fundamental tensor, \emph{Amer. J. Math.} 96 (1974), 207--213.

\bibitem{O-N} B. O'Neill, \emph{Semi-Riemannian Geometry with applications to Relativity}, Academic Press, 1983.

\end{thebibliography}
\end{document}